\documentclass[12pt]{amsart}
\usepackage{xfrac} 
\usepackage{fullpage}
\usepackage{calrsfs}
\usepackage{amssymb, comment, color}
\usepackage[toc,page]{appendix}

\usepackage{epsf,overpic,amssymb,amscd,times,euscript}
\usepackage{faktor}
\usepackage[utf8]{inputenc}
\usepackage[english]{babel}
\usepackage{amsmath}
\usepackage{amsthm}
\usepackage{amsfonts}
\usepackage{amssymb}
\usepackage{mathtools}
\usepackage{amscd}
\usepackage{bezier}
\usepackage{enumerate}
\usepackage{siunitx}
\usepackage{url}

\usepackage{breakurl}
\usepackage[breaklinks]{hyperref}

\usepackage{graphicx}

\usepackage{upgreek}
\numberwithin{equation}{section}

\usepackage[vmargin=1in,hmargin=1in]{geometry}
\usepackage{enumitem}
\setlist[description]{style=nextline}
\setlist[itemize,enumerate]{leftmargin=*,itemsep=0pt,parsep=0pt}
\emergencystretch=\maxdimen
\usepackage{parskip}

\usepackage{bbm}

\newcommand{\N}{\mathbbm{N}}                     
\newcommand{\Z}{\mathbbm{Z}}                     
\newcommand{\R}{\mathbbm{R}}                     

\newcommand{\D}{\mathbb{D}}

\newcommand{\Diff}{\mathrm{Diff}}               

\newtheorem{thm}{Theorem}[section]               
\newtheorem*{thm*}{Theorem}               
\newtheorem{cor}[thm]{Corollary}        
\newtheorem*{cor*}{Corollary}        
\newtheorem{lem}[thm]{Lemma}  
\newtheorem*{lem*}{Lemma}
\newtheorem{prop}[thm]{Proposition}     
\theoremstyle{definition}
\newtheorem{defn}[thm]{Definition}      
\newtheorem{rem}[thm]{Remark}           
\newtheorem{ques}[thm]{Question}   
 \newtheorem*{acknowledgement*}{\protect\acknowledgementname}
\newcounter{claim}

\newenvironment{claimproof}[1]{\par\noindent\underline{Proof:}\space#1}{\qed}

 \providecommand{\acknowledgementname}{Acknowledgement}

\binoppenalty=\maxdimen
\relpenalty=\maxdimen

\author{David Bechara Senior}
\address{David Bechara Senior, RWTH Aachen, Jakobstrasse 2, Aachen 52064, Germany}
\email{\texttt{bechara@mathga.rwth-aachen.de}}

\author{Patrice Le Calvez}
\address{Patrice Le Calvez, Institut de Math\'ematiques de Jussieu-Paris Rive Gauche, IMJ-PRG, Sorbonne Universit\'e, Universit\'e Paris-Cit\'e, CNRS, F-75005, Paris, France}
\email{\texttt{patrice.le-calvez@imj-prg.fr}}
\author{Abror Pirnapasov}
\address{Abror Pirnapasov, University of Maryland, 4176 Campus Drive - William E. Kirwan Hall
College Park, MD 20742}
\email{\texttt{abrorpirnapasov@gmail.com}}

\title[The asymptotic mean action and the asymptotic linking number for pseudo-rotations]{The asymptotic mean action and the asymptotic linking number for pseudo-rotations}
\begin{document}

\begin{abstract}
We establish that the asymptotic mean action and the asymptotic linking number of irrational pseudo-rotations remain well-defined everywhere and constant for every $C^{1}$ irrational pseudo-rotation that behaves as a rotation on the boundary. As a consequence, we demonstrate that the isotopy of irrational pseudo-rotations with a positive rotation number is a right-handed isotopy in the sense of Ghys. 
\end{abstract}

\maketitle
\section{Introduction and Main Results}
\subsection{The asymptotic mean action of an irrational pseudo-rotation}

We consider $\mathbb{D}\subset \R^2
$ the unit closed two disk centered at the origin with polar coordinates $(r,\theta)$. Let $f$ be an area-preserving smooth diffeomorphism of $\D$  isotopic to $\mathrm{id}$ via an area-preserving isotopy $\tilde{f}=\{f_t\}_{t\in [0,1]}$. Let $\beta$ be a primitive of the standard symplectic form $\omega=\frac{r}{\pi}dr\wedge d\theta.$ 

The \textit{action} of $f$ with respect to $\beta$ and the given isotopy is the unique smooth function \sloppy $a_{\tilde{f}}(x):\mathbb{D}\to \mathbb{R}$ satisfying:
\begin{equation}\label{boundary123}
    \begin{cases}
        da_{\tilde{f}}=f^{*}\beta-\beta\\
        a_{\tilde{f}}(x)=\int_{\tilde{f}(x)}\beta \qquad \text{for every} \  x\in \partial\mathbb{D},
    \end{cases}
\end{equation}
where $\tilde{f}(x)$ is the path $t\mapsto f_{t}(x)$.

 We define the corresponding Calabi invariant by
\[
\mathrm{CAL}(\tilde{f}):=\int_{\mathbb{D}}a_{\tilde{f}} \  \omega.
\]
{The quantity $a^{\infty}_{\tilde{f}}(x)$ is called the \textit{asymptotic mean action of 
 $x$ with respect to $f$}, whenever  following limit exists}
\begin{equation*}
    a^{\infty}_{\tilde{f}}(x):=\lim_{n\rightarrow \infty}\frac{1}{n}\sum_{i=0}^{n-1}a_{\tilde{f}}({f}^{i}(x)).
\end{equation*}
A straightforward application of Birkhoffs' ergodic theorem implies that $a^{\infty}_{\tilde{f}}(x)$ is well defined for Lebesgue a.e. $x\in\mathbb{D}$, furthermore it is independant of the primitive $\beta$ of $\omega$. If $x$ is periodic or  $x\in\partial \mathbb{D}$, then  
$a^{\infty}_{\tilde{f}}(x)$ is well defined. We will denote $\mathrm{Per}(f)$ the set of periodic points of $f$. Moreover $a^{\infty}_{\tilde{f}}(x)$ is equal to $\rho(\tilde f\big|_{\partial \mathbb{D}})$ for every $x \in \partial\D$ where $\rho(\tilde f\big|_{\partial \mathbb{D}})$ is the rotation number of the isotopy $\tilde {f}\big| _{\partial \mathbb D}=\{ f_t\big\vert_{\partial \mathbb{D}}\}_{t\in [0,1]}$ restricted to the boundary of the disc $\partial \mathbb{D}$.

In \cite{hutchings2016mean} Hutchings discovered a very interesting relation between the Calabi invariant and the asymptotic mean action of periodic points of smooth area-preserving disk maps which are a rotation on the boundary.
\begin{thm}\cite{hutchings2016mean}\label{hut}
     Let $f$ be a smooth area-preserving diffeomorphism of $\mathbb{D}$
that is a rigid rotation near the boundary and isotopic to $\mathrm{id}$ via $\tilde{f}$. If 
\begin{align}\label{1.2}
\mathrm{CAL}(\tilde{f}) < \rho(\tilde f\big|_{\partial \mathbb{D}}), 
\end{align}
then
\begin{align}\label{1.3}
\inf \left\{  a^{\infty}_{\tilde{f}}(x)\,|\, x \in \mathrm{Per}(f)\right\}\leq \mathrm{CAL}(\tilde{f}).
\end{align}
\end{thm}
Weiler \cite{weiler2021mean} established a similar result for area-preserving diffeomorphisms of the closed annulus. The third author \cite{Pirnapasov2021HutchingsIF} proved Theorem \ref{hut} without the assumption of the map being a rotation near the boundary by using preliminary extensions. The third author and Prasad \cite{PirPra} proved the validity of inequality (\ref{1.3}) for generic Hamiltonian maps of surfaces with boundaries by proving the generic equidistribution theorem for Hamiltonian diffeomorphisms of surfaces with boundary. Recently, in \cite{nelson2023torus}, Nelson and Weiler established a nongeneric version of Theorem \ref{hut} on higher-genus surfaces but with stronger assumptions. 

For the particular case of irrational pseudo-rotations the third author showed that the Calabi invariant of smooth irrational pseudo-rotations is equal to the rotation number(cf.\cite{Pirnapasov2021HutchingsIF}), this was proven by using the general case of Theorem \ref{hut}. This result was subsequently reaffirmed by the second author using finite-dimensional methods for $C^{1}$ irrational pseudo-rotations which are conjugated to a rotation on the boundary.

The following question, posed in \cite{senior2020relation} in the context of three-dimensional Reeb dynamics, addresses the relationship between the asymptotic mean action of periodic points and the asymptotic mean action of all points:
\begin{ques}\label{ques}
Does the following equality hold for all area-preserving disk maps?
\begin{equation}\label{ques1}
    \overline{\big\{a^{\infty}_{\tilde{f}}(x): \ x\in\mathbb{D}\big\}}=   \overline{\big\{a^{\infty}_{\tilde{f}}(x): \ x\in \mathrm{Per}(f) \big\}}
\end{equation}
\end{ques}
Equality (\ref{ques1}) implies Theorem \ref{hut} because the set $\overline{\big\{a^{\infty}_{\tilde{f}}(x): \ x\in\mathbb{D}\big\}}$ contains the Calabi invariant. In his thesis \cite{BecharaSenior2023}, the first author showed that the equality (\ref{ques1}) holds for all diffeomorphisms generated by an autonomous Hamiltonian of the disc.
\begin{defn}
    An area-preserving  diffeomorphism $f:\mathbb{D}\to \mathbb{D}$ is called an irrational pseudo-rotation if $\#\mathrm{Per}(f)=1$.
\end{defn}
In the following theorem, we affirmatively answer question \ref{ques} for $C^{1}$ irrational pseudo-rotations, which are conjugate to a rotation when restricted to the boundary.
\begin{thm}\label{asymp-mean}
    If $f:\mathbb{D} \to \mathbb{D}$ is a $C^{1}$ irrational pseudo-rotation that is $C^{1}$ conjugate to a rotation on the boundary, then the sequence of functions 
    $$\frac{\sum_{i=0}^{n-1}a_{\tilde{f}}({f}^{i}(x))}{n}$$ uniformly converges to $\rho(\tilde{f})$. In particular, the asymptotic mean action $a^{\infty}_{\tilde{f}}(x)={\rho(\tilde{f})}$, where $\rho(\tilde{f})$ is the Poincare rotation number of the isotopy restricted to the boundary of the disk.
\end{thm}
\begin{rem}
We prove Theorem \ref{asymp-mean} for irrational pseudo-rotations that fix $0$ and coincide with a rotation on the boundary. This is merely a technical condition because if $x\in \mathbb{D}$ is a fixed point of an irrational pseudo-rotation $f$, then there exists a compactly supported Hamiltonian $g$ such that $g(0)=x$. This implies that $g\circ f\circ g^{-1}$ is also an irrational pseudo-rotation and $0$ is its fixed point.  We also note that every $C^1$ diffeomorphism of $\partial {\mathbb{D}}$ can be extended as a $C^1$ area preserving diffeomorphism of $ \mathbb{D}$. This is the reason why we can suppose that $f$ is $C^1$ conjugate to a rotation on $\partial \mathbb{D}$ in Theorem 1.3.  
\end{rem}
 In his thesis \cite{BecharaSenior2023}, the first author proved Theorem \ref{asymp-mean} for $C^{0}$-rigid irrational pseudo-rotation using Joly's winding number bound [Lemma 5.3 in\cite{beno}] for $C^{0}$-rigid irrational pseudo-rotations. This statement also holds for minimally ergodic irrational pseudo-rotations \cite{fayad2004constructions}, which are irrational pseudo-rotations having only three ergodic invariant measures: the Dirac measure at the fixed point, the Lebesgue measure on the boundary, and the Lebesgue measure on the disc. Together with the Calabi identity result in \cite{Pirnapasov2021HutchingsIF}, minimally ergodicity implies that the integral of the action function with respect to any invariant probability measure is equal to the rotation number. This implication leads to the conclusion of Theorem \ref{asymp-mean} using a Krylov-Bogolyubov type argument. In the proof of Theorem \ref{asymp-mean}, we also prove that the integral of the action function with respect to any invariant measure is equal to the rotation number using Le Calvez's finite-dimensional methods.
\subsection{The asymptotic orbit linking number of an irrational pseudo-rotation}
Let $f\in \text{Diff}^{1}(\D, \omega_0)$  and let $\tilde{f}=\{f_t\}_{t\in[0,1]}$ be a Hamiltonian isotopy such that $f^0=\mathrm{id}$ and $f^1=f$. Here, $\text{Diff}^{1}(\D, \omega_0)$ is the set of $C^{1}$ diffeomorphisms of $\mathbb{D}$ which preserves $\omega_{0}$.
\begin{defn}\label{DefWindingNumber}
The \textit{winding number} $W_{\tilde{f}}(x, y)$ of a pair of distinct  points $x,y\in \D$ is the real number
\[
W_{\tilde{f}}(x, y):=\frac{{\theta}(1)-{\theta}(0)}{2\pi},
\]
where $\theta: [0,1] \rightarrow \R$ is a continuous function such that
\begin{equation}\label{star}
 \frac{f_t(y)-f_t(x)}{||f_t(y)-f_t(x)||}=e^{i\theta(t)}.
\end{equation}
\end{defn}
The value of $W_{\tilde{f}}(x, y)$ does not depend on the isotopy joining $f$ to the identity if isotopies have the same rotation number on the boundary. This holds because any two paths joining $f$ to the identity are homotopic if the isotopies have the same rotation number when restricted to the boundary, which implies that the winding number is the same. The map $W_{\tilde{f}}$ is smooth and bounded on $(\D\times\D)\setminus \Delta$, where $\Delta$ denotes the diagonal $\{(x, x)|x\in \D \} $ in $\D\times\D$. To see this, let $K$ be indeed the blow-up of the diagonal in $\D \times \D$. In more concrete terms, $K$ is the compact set formed by replacing $\Delta $ with the unit tangent bundle $T^1\D$). Of course $\D \times \D \setminus \Delta $ naturally embeds into $K$ as a dense open set. To be more precise a sequence $(x_n,y_n)\subset \D\times \D \setminus \Delta$ converges to a point $(x,\xi)\in T^1\D\simeq \mathbb{D}\times S^{1}$ whenever $x_n,y_n\rightarrow x$ and if the sequence of oriented vectors $\frac{y_n-x_n}{||y_n-x_n||}=e^{i\theta_n}\in \mathbb{S}^1$ converges to $\xi$. We assert that if $f$ is of class $C^1$, the function $W_{\tilde{f}}$ extends to $K$ as a continuous function, and in particular, $W_{\tilde{f}}$ is bounded. This is, in fact, an extension of the definition of $W_{\tilde{f}}$ to element the form $(x,\xi)\in T^1\D$. Let $(f_t)$ be the isotopy between the identity and $f$ among the $C^1$-diffeomorphisms of $\text{Diff}^1(\D, \omega_0)$. We then define $W_{\tilde{f}}(x,\xi)$ as the variation in the angle of the vector $df_t(\xi)\in \R^2$, images of $\xi$ under the differential $df_t$, as $t$ varies from $0$ to $1$. This indeed defines the desired continuous extension of $W_{\tilde{f}}$ to $K$.

Define the set $\mathbb{O}(\mathbb{D}) \subset \mathbb{D}\times\mathbb{D}$ as follows:
$$ \mathbb{O}(\mathbb{D})=\{(x,y)\in{\mathbb D}\times{\mathbb D}\,\vert \enskip f^i(x)\not=f^{j}(y) \enskip \mathrm{for\enskip every\enskip} (i,i')\in\mathbb{Z}\times \mathbb{Z}\}.$$
\begin{defn}
   Suppose $(x,y)\in \mathbb{O}(\mathbb{D})$, the number 
   \[
   L_{\tilde{f}}^{\infty}(x,y):=\lim_{n\to \infty}\frac{\sum_{i=0}^{n-1}\sum_{j=0}^{n-1}W_{\tilde{f}}(f^{i}(x),f^{j}(y))}{n^{2}}
   \]
   is called  the \textit{asymptotic orbit linking number} of $(x,y)$ if the limit exists.
\end{defn}
In the next theorem, we prove that the asymptotic linking number exists and is equal to the rotation number of the isotopy for every pair of points that are not in the same orbit of the map.
\begin{thm}\label{asym-link}
    If $f:\mathbb{D} \to \mathbb{D}$ is a $C^{1}$ irrational pseudo-rotation that is $C^{1}$ conjugate to a rotation on the boundary and $\tilde f$ is an isotopy from $\mathrm{id}$ to $f$, then   the sequence   
    \[
   \left (\frac{\sum_{i=0}^{n-1}\sum_{j=0}^{n-1}W_{\tilde{f}}(f^{i}(x),f^{j}(y))}{n^{2}}\right)_{n\geq 0}
    \]
    converges uniformly  to $\rho(\tilde{f})$ for every $(x,y)\in \mathbb{O}(\mathbb{D})$. In particular,  
    the asymptotic mean linking number $L_{\tilde{f}}^{\infty}(x,y)$ is defined for every $(x,y)\in \mathbb{O}(\mathbb{D})$ and equal to $\rho(\tilde{f})$. 
\end{thm}
\begin{defn}\label{defn:right} The isotopy $\tilde{f}$ is called a right-handed isotopy, if 
\begin{itemize}
\item for every $(x,y)\in \mathbb{O}(\mathbb{D})$ if holds that
\[\liminf_{n\to \infty}\frac{1}{n^2}\sum_{i=0}^{n-1}\sum_{j=0}^{n-1}W_{\tilde{f}}(f^{i}(x),f^{j}(y)) >0.\]
\item for every $x\in \mathrm{Per}(f)$ if holds that
\[\liminf_{n\to\infty}\frac{1}{n^2}\left(\sum_{\substack{i,j=0 \\ f^{i}(x)\neq f^{j}(x)}}^{n-1}W_{\tilde{f}}(f^{i}(x),f^{j}(x)) + \sum_{\substack{i,j=0 \\ f^{i}(x)= f^{j}(x)}}^{n-1}W_{\tilde{f}}(f^{i}(x),df^j(\xi))\right)>0.\]
\end{itemize}
\end{defn}
\begin{rem}
  The isotopy $\tilde{f}$ is called left-handed if both limits are negative in Definition \ref{defn:right}.
\end{rem}
In \cite{ghys2009right}, Ghys introduced right-handed flows and demonstrated that any finite collection of periodic orbits admits a global surface of section. Subsequently, in \cite{florio2021quantitative}, Florio and Hryniewicz provided a sufficient condition for Reeb flows to be right-handed. Prasad \cite{Prasad2022VolumepreservingRV} showed that volume-preserving right-handed flows are conformally equivalent to  Reeb flows. In her thesis \cite{lhuissier2018mathematical}, Lhuissier showed that definition \ref{defn:right} is equivalent to the definition of right-handedness due to Ghys if one looks at the suspension of the disk map as a flow (see \cite{lhuissier2018mathematical} Theorem 7.3.2).
\begin{cor}\label{right-hand}
   If $f:\mathbb{D} \to \mathbb{D}$ is a $C^{1}$ irrational pseudo-rotation that is $C^{1}$ conjugate to a rotation on the boundary, and $\tilde f$ is an area-preserving isotopy from the identity map $\mathrm{id}$ to $f$ such that the rotation number of the isotopy is (negative) positive, then the isotopy $\tilde f$ is (left-handed) right-handed.
\end{cor}
Finally we proceed to prove Corollary \ref{right-hand} using Theorem \ref{asym-link}.
\begin{proof}[Proof of Corollary \ref{right-hand}] It is enough to prove the theorem in the case were the isotopy is right-handed. Assume that the  rotation number of the isotopy is positive. By using Theorem \ref{asym-link}, we can conclude that for every $(x,y)\in \mathbb{O}(\mathbb{D})$, the asymptotic orbit linking number is defined and equal to the rotation number and by assumption, it is positive. That is
\[
L_{\tilde{f}}^{\infty}(x,y)=\lim_{n\to \infty}\frac{\sum_{i=0}^{n-1}\sum_{j=0}^{n-1}W_{\tilde{f}}(f^{i}(x),f^{j}(y))}{n^{2}}=\rho(\tilde{f})>0
\]
for every $(x,y)\in \mathbb{O}(\mathbb{D})$. 

Now we check the second condition of Definition \ref{defn:right}. Since the fixed point $0$ of $f$ is the only periodic point of $f$, the sum in the second condition simplifies to the following sum:
$$\liminf_{n\to \infty}\frac{1}{n}\sum_{i=0}^{n-1}W_{\tilde{f}}(0,df^i(\xi)) >0.$$ 
But, this is the linearized rotation number of $0$ through the isotopy. By the classical result of Franks' \cite{franks1988generalizations} and also of the second author \cite{le2001rotation}, the linearized rotation number of the fixed point is also equal to $\rho(\tilde{f})$ hence it is also positive.

This completes the proof of Theorem \ref{right-hand}. 
\end{proof}
{\bf Structure of the article:} 
In Section 2, we recall basic notions relevant to the article, such as weak convergence of probability measures, radial foliations, and their properties. In Section 3, we introduce radial foliations and isotopies for irrational pseudo-rotations whose study has been developed by the second author in \cite{Cal}; then, we show that the measure between any leaf of a radial foliation and its image is equal to the rotation number of the map. By combining this with properties of the foliation, we show that the integral of the winding number with respect to the product of any two invariant probability measures is also equal to the rotation number. In Section 4, we use the integral of the winding number together with a Krylov-Bogolyubov argument to prove Theorem \ref{asymp-mean} and Theorem \ref{asym-link}.
 \subsection*{Acknowledgments} AP would like to express his gratitude to Barney Bramham for a lot of conversations and suggestions about asymptotic mean action. AP is also grateful to Alberto Abbondandolo, Umberto Hryniewicz, Marco Mazzuchelli, and Rohil Prasad for many stimulating conversations and
suggestions. The work of DBS is partially supported by the DFG SFB/TRR 191 ``Symplectic Structures in Geometry, Algebra and Dynamics'', Projektnummer 281071066-TRR 191. The work of AP is partially supported by the DFG Walter Benjamin Fellowship, Projektnummer 518128580, and the ANR CoSyDy ``Conformally Symplectic Dynamics beyond symplectic dynamics'' (ANR-CE40-0014). 
\section{Preliminaries}
In this section, we introduce necessary results in ergodic theory and also in the theory of radial foliations specifically in the case of irrational pseudo-rotations. In order to state several of the inequalities needed to prove the main theorems, we require different quantities coming from measurements of topological angles.
\subsection{Invariant measures and pointwise ergodic theorem.} Assume $f:\mathbb{D}\to\mathbb{D}$ is a homeomorphism. For $x\in \mathbb{D}$, let $\delta_x$ denote the Dirac measure supported at $x$. We will denote:
\[
\mu_{n}(x):=\frac{1}{n}\sum^{n-1}_{i=0}\delta_{f^{i}(x)}
\]
\begin{prop}\label{invmeas1} Let $(x_{k})_{k\geq 0}$ and $(y_{k})_{k\geq 0}$ be two sequences in $\mathbb{D}$ and $(n_{k})_{k\geq 0}$  an increasing sequence of positive integers. Then there exists a subsequence $(n_{k_{j}})_{j\geq 0}$ such that the sequences of measures $(\mu_{n_{k_{j}}}(x_{k_{j}}))_{j\geq 0}$ and $(\mu_{n_{k_{j}}}(y_{k_{j}}))_{j\geq 0}$ converge for the weak* topology and their limits are invariant probability measures of $f$.
\end{prop} 
The proof of Proposition \ref{invmeas1} is very similar to the proof of the Krylov-Bogolyubov Theorem and is based on the compactness of the space of Borel probability measures furnished with the  weak$^*$ topology and the fact that every limit point of a sequence $(\mu_{n_{k}}(x_{k}))_{k\geq 0}$ is invariant.
The next result is straightforward.
\begin{prop} \label{weakcon}
   If sequences of probability measures $(\mu_k)_{k\geq 0}$ and $(\nu_k)_{k\geq 0}$ converge for the weak$^*$ topology to $\nu$ and $\mu$, then the sequence of probability measures $(\mu_k\times \nu_k)_{k\geq 0}$  converges for the weak$^*$ topology to $\nu\times\mu$.
\end{prop}
The following theorem is a direct implication of Templemann's pointwise ergodic theorem. For more information, please refer to Theorem 2.6 in Section 2.3.3 of \cite{sarig2009lecture}.
\begin{thm}
Assume that $\mu$ and $\nu$ are invariant probability measures of $f$. If $\phi\in L^{1}(\mathbb{D}\times \mathbb{D}, \mu\times \nu)$, then 
\[
S_{n}\phi(z):=\frac{1}{n^{2}}\sum_{i=0}^{n-1}\sum_{j=0}^{n-1}\phi(f^{i}(x),f^{j}(y))
\]
converges for $(\mu\times \nu)$-almost every $z=(x,y)\in \mathbb{D}\times \mathbb{D}$.
\end{thm}
\subsection{Properties of radial foliations and topological angles}
Let us begin with some notations and definitions, specially regarding radial foliations on a disk. Let $D:=\{z\in\R^2\,\vert\enskip \Vert z\Vert <r\}$ be an open disk and set $D^*=D\setminus \{0\}$ the punctured disk, lastly let us write $\Tilde{D}$ for the universal covering space of $D^*$. We will consider the sets
$$W=\{(z,z')\in D^{*}\times D^{*} \ | \ z\neq z'\}$$
and
$$\tilde{W}=\{(\tilde z,\tilde z')\in \tilde{D}\times \tilde{D} \ | \ \tilde z\neq \tilde z'\}.$$
\begin{defn}
    A {\it radial foliation } is an oriented topological foliation of $D^*$ such that every leaf $\phi$ is a {\it ray}, meaning that it satisfies
    \[
    \alpha(\phi)=\{0\}, \qquad \omega(\phi) \subset \overline{D}\setminus D \quad\text{(here the closure is taken in } \mathbb{R}^2)    \]
    In other words, $\phi$ tends to $0$ in the past and to $\partial D$ in the future. We denote ${\mathfrak F}$ the set of radial foliations. The group $\mathrm{Homeo}_*(D^*)$ of homeomorphisms of $D^*$ that are isotopic to the identity, acts naturally on ${\mathfrak F}$: if $\mathcal F\in\mathfrak F$ and $f\in\mathrm{Homeo}_*(D^*)$, then the foliation $f(\mathcal F)$, whose leaves are the images by $f$ of the leaves of $\mathcal F$, is a radial foliation. Moreover, $\mathrm{Homeo}_*(D^*)$ acts transitively on $\mathfrak F$.  If $\mathrm{Homeo}_*(D^*)$ is endowed with the $C^0$ topology, meaning the compact-open topology applied to maps and the inverse maps, the stabilizer $\mathrm{Homeo}_{*,\mathcal F}(D^*)$ of $\mathcal F\in \mathfrak F$ is a closed subgroup of  $\mathrm{Homeo}_*(D^*)$. Consequently, the map
\[
\begin{aligned}\mathrm{Homeo}_*(D^*)/\mathrm{Homeo}_{*,\mathcal F}(D^*)&\to {\mathfrak F}\\ f \,\mathrm{Homeo}_{*,\mathcal F}(D^*)&\mapsto f(\mathcal F)
 \end{aligned}
 \]
 is bijective and ${\mathfrak F}$ can be furnished with a natural $C^0$ topology. It is the topology, induced from the quotient topology defined on $\mathrm{Homeo}_*(D^*)/\mathrm{Homeo}_{*,\mathcal F}(D^*)$ by this identification map. Note that this topology does not depend on $\mathcal F$.
 It is well known that the fundamental groups of $\mathrm{Homeo}_*(D^*)$ and $\mathrm{Homeo}_{*,\mathcal F}(D^*)$ are infinite cyclic and that the morphism $i_* :\pi_1(\mathrm{Homeo}_{*,\mathcal F}(D^*), \mathrm{Id})\to \pi_1(\mathrm{Homeo}_*(D^*), \mathrm{Id})$  induced by the inclusion map $i:  \mathrm{Homeo}_{*,\mathcal F}(D^*)\to \mathrm{Homeo}_*(D^*)$ is bijective. It implies that ${\mathfrak F}$ is simply connected. 
In the whole article, when $\mathcal F$ is a radial foliation, the notation $\tilde{\mathcal F}$ means the lift of $\mathcal F$ to the universal covering space $\tilde D$. Note that $\tilde{\mathcal F}$ is a trivial foliation. In particular, every leaf $\tilde \phi$ of $\tilde{\mathcal F}$  is an oriented line of $\tilde D$ that separates $\tilde D$ into two connected open sets, the component of  $\tilde D\setminus\tilde\phi$ lying on the left of $\tilde \phi$ will be denoted $L(\tilde \phi)$ and the component lying on the right will be denoted $R(\tilde \phi)$. We get a total order $\preceq_{\mathcal F}$ on the set of leaves of $\tilde {\mathcal F}$ as follows:
$$\tilde \phi \preceq_{\mathcal F}  \tilde \phi'\Longleftrightarrow R(\tilde\phi)\subset R(\tilde\phi').$$
We can also define a partial order $\leq_{ {\mathcal F}}$ on $\tilde D$: denote $\tilde\phi_{\tilde z}$ the leaf of $\tilde{\mathcal F}$ that contains $\tilde z$ and write $\tilde z<_{{\mathcal F}}\tilde z'$ if $\tilde z\not =\tilde z'$, if $\tilde\phi_{\tilde z}=\tilde\phi_{\tilde z'}$ and if the segment of $\tilde\phi_{\tilde z}$ beginning at $\tilde z$ and ending at $\tilde z'$ inherits the orientation of  $\tilde\phi_{\tilde z}$.
\end{defn}
The above formalism of radial foliations allows us to work with a certain ``discretization of angles", these ideas were introduced by the second author in \cite{Cal} and further explained in \cite{calvez2023twist} with specific applications to twist maps. For the purpose of this article we will introduce only the necessary definitions without stepping into all of the many properties of this technique. We now introduce topological angles, these serve as a bridge between the measurement given by the winding number \ref{star} and the information captured by radial foliations.\\\\
Consider the natural projection
\[
\begin{aligned}\pi: \Z&\to \Z/4\Z\\
 k&\mapsto \dot k=k+4\Z.
 \end{aligned}
 \]
 If $\Z/4\Z$ is endowed with the topology whose open sets are
\[
\emptyset, \{\dot 1\}, \{\dot 3\}, \{\dot 1, \dot 3\}, \{\dot 1, \dot 2, \dot 3\}, \{\dot 3, \dot 0,\dot 1\},\{\dot 0,\dot 1, \dot 2, \, \dot 3\},
\]
and $\Z$ with the topology generated
 by the sets $2k+1$ and $\{2k+1, 2k, 2k+1\}$, $k\in\Z$, then $\pi$ is a covering map\footnote{This topology is usually called the {\it digital line topology} on $\Z$. It is also known as {\it Khalimsky topology}}. Note that both sets
$\Z/4\Z$ and $\Z$ are non Hausdorff but path connected. 
If $k$, $l$ are two integers, we will define
\[
\lambda(k,l)=\begin{dcases}0 &\text{ if $k=l$,}\\
\#(k,l)\cap 4\Z +  (\#\{k,l\}\cap 4\Z)/2 &\text { if $k<l$,}\\
-\#(l,k)\cap 4\Z - (\#\{l,k\}\cap 4\Z)/2 &\text{ if $k>l$.}\end{dcases}
\]
 Note that for every integers $k, l, m$ we have 
 \[
\lambda(k,l)+\lambda(l,m)=\lambda(k,m).
 \]
The quantity $\lambda(k,l)$ measures the ``algebraic intersection number'' of a continuous path $\gamma: [s_0,s_1]\to \Z/4\Z$ with $\{\dot 0\}$, if $\gamma$ is lifted to a path $\hat\gamma: [s_0,s_1]\to\Z$ joining $k$ to $l$. 
For every $(\tilde z, \tilde z')\in \tilde W$, we can define a continuous function $\dot{\theta}_{\tilde z, \tilde z'}: {\mathfrak F}\to \Z/4\Z$ as follows:
\[
\dot{\theta}_{\tilde z, \tilde z'}( \mathcal F)= \begin{cases} \dot 0 &\text{ if $\tilde z'>_{\mathcal F}\tilde z$,}\\
\dot 1 &\text{ if $\phi_{\tilde z'} \succ_{\mathcal F} \phi_{\tilde z}$,}\\
 \dot 2 &\text{ if  $\tilde z'<_{\mathcal F}\tilde z$,}\\
 \dot  3 &\text{ if $\phi_{\tilde z'} \prec_{\mathcal F} \phi_{\tilde z}$.}\end{cases}
 \]
 The space $\mathfrak F$ being simply connected, the Lifting Theorem asserts that there exists a continuous function, 
$\theta_{\tilde z, \tilde z'}: {\mathfrak F}\to \Z$, uniquely defined up to an additive constant in $4\Z$, such that $\pi\circ \theta_{\tilde z, \tilde z'}= \dot{\theta}_{\tilde z, \tilde z'}.$
Note that
\begin{itemize}
\item $\dot{\theta}_{\tilde z', \tilde z}( \mathcal F) =\dot{\theta}_{\tilde z, \tilde z'}( \mathcal F)+\dot 2$, for every $\mathcal F\in \mathfrak F$,
\item $\theta_{\tilde z, \tilde z'}$ and $\theta_{\tilde z', \tilde z}$ can be chosen such that $\theta_{\tilde z', \tilde z}( \mathcal F) =\theta_{\tilde z, \tilde z'}( \mathcal F)+2$ for every $\mathcal F\in \mathfrak F$,
\item  $\dot{\theta}_{\tilde f(\tilde z), \tilde f(\tilde z')}( f(\mathcal F)) =\dot{\theta}_{\tilde z, \tilde z'}( \mathcal F) $, for every $\mathcal F\in \mathfrak F$ and every $f\in\mathrm{Homeo}_*(D^*)$,
\item $\theta_{\tilde f(\tilde z), \tilde f(\tilde z')}$ and $\theta_{\tilde z, \tilde z'}$ can be chosen such that $\theta_{\tilde f(\tilde z), \tilde f(\tilde z')}( f(\mathcal F)) =\theta_{\tilde z, \tilde z'}( \mathcal F) $, for every $\mathcal F\in \mathfrak F$.
\end{itemize}

In particular, if $\mathcal{F}$ and $\mathcal{F}'$ are two radial foliations, the numbers
\[
{\tau}(\tilde z,\tilde z',\mathcal{F},\mathcal{F}'):={\theta}_{ \tilde z,\tilde z'}(\mathcal{F}')-{\theta}_{\tilde z,\tilde z'}(\mathcal{F})
\]
and 
\[
\lambda(\tilde z,\tilde z',\mathcal{F},\mathcal{F}'):=\lambda({\theta}_{\tilde z,\tilde z'}(\mathcal{F}'),{\theta}_{\tilde z,\tilde z'}(\mathcal{F}))
\]
do not depend on the choice of the lift ${\theta}$.

We can define a pseudo-distance on $\mathfrak F$, by writing $d(\mathcal
F, \mathcal F')=\sup_{(\tilde z, \tilde z')\in\tilde W}  \tau(\tilde
z,\tilde z', \mathcal F, \mathcal {F'})\in [0,+\infty]$. We call it the {\it winding distance}.
\subsubsection{ Topological angles in the annulus}
Let $\mathcal F$ be a radial foliation and $f$ a homeomorphism of $D^*$ isotopic to the identity. Let $I=(f_s)_{s\in[0,1]}$ be an identity isotopy of $f$ in $\mathrm{Homeo}_*(D^*)$ and $\tilde I=(\tilde f_s)_{s\in[0,1]}$ the lifted identity isotopy to $\tilde D$.  
The function
$$s\in[0,1]\mapsto \dot{\theta}_{\tilde f_s(\tilde z), \tilde f_s(\tilde z')} (\mathcal F)=\dot{\theta}_{\tilde z,\tilde z'} ( f_s^{-1} (\mathcal F))\in \Z/4\Z$$ can be lifted to a function $$s\in[0,1]\mapsto \theta_{\tilde z,\tilde z'} ( f_s^{-1} (\mathcal F))\in\Z$$ and the difference between the value at $1$ and the value at $0$ of this last map is nothing but $ \tau (\tilde z, \tilde z, {\mathcal F}, f^{-1}(\mathcal F))$. Of course, it depends neither on the choice of $I$, nor on the choice of the lift $\tilde I$ of $I$. Suppose now that $\tilde z$ and $\tilde z'$ project onto two different points of $D^*$. Note that if $k$ is large enough, then for every $s\in[0,1]$, it holds that
$$\dot{\theta}_{\tilde f_s(\tilde z), \tilde f_s(T^k(\tilde z'))} (\mathcal F)=\dot 1, \enskip \dot{\theta}_{\tilde f_s(\tilde z), \tilde f_s(T^{-k}(\tilde z'))} (\mathcal F)=\dot 3,$$ and so the functions
$$ s\mapsto \theta_{ \tilde z,  T^k(\tilde z')} (f_s^{-1}(\mathcal F)), \enskip  s\mapsto  \theta_{\tilde z, T^{-k}(\tilde z')}(f_s^{-1}(\mathcal F))$$ are constant,  which means that 
$$ \tau (\tilde z, T^k(\tilde z'), \mathcal F,f^{-1}( {\mathcal F}) )= \tau (\tilde z, T^{-k}(\tilde z'), \mathcal F, f^{-1}({\mathcal F})) =0.$$
Consequently, if $\mathcal F$ and $\mathcal F'$ are two radial foliations, then for every $(z,z')\in W$, one can define
\[
\begin{aligned} \overline{ \tau} (z, z', \mathcal F, {\mathcal F}') &:=\sum_{k\in\Z} \vert  \tau (\tilde z, T^k(\tilde z'), \mathcal F, {\mathcal F}') \vert,\\
  \tau (z, z', \mathcal F, {\mathcal F}') &:=\sum_{k\in\Z} \tau (\tilde z, T^k(\tilde z'), \mathcal F, {\mathcal F}'), \\
\lambda(z, z', \mathcal F, {\mathcal F}') &:=\sum_{k\in\Z} \lambda(\tilde z, T^k(\tilde z'), \mathcal F, {\mathcal F}') ,\end{aligned}
\]
each sum being independent of the choice of lifts $\tilde z, \tilde z'$ of $z, z'$. These identities implies the following inequality
\begin{equation}\label{inequality}
\vert\lambda(z,z', \mathcal F, \mathcal F') \vert \leq\overline{\tau}(z,z', \mathcal F, \mathcal F').
\end{equation}
Now let us define what is a {\it displacement function}. For every $f\in\mathrm{Homeo}_*(D^*)$, every lift $\tilde f$ of $f$ to $\tilde D$ and every ray $\phi$ we can define a function $m_{\tilde f, \phi}: D^*\to\Z$ as follows:
we choose a lift $\tilde \phi$ of $\phi$, then for every $z\in D$, we consider the lift $\tilde z$ of $z$ such that $\tilde z\in\overline {L(\tilde\phi)}\cap R(T(\tilde\phi))$ and we denote $m_{\tilde f, \phi}(z)$ the integer $m$ such that  $\tilde f(\tilde z)\in\overline {L(T^m(\tilde\phi))}\cap R(T^{m+1}(\tilde\phi))$, noting that it does not depend on the choice of $\tilde \phi$. Observe that $m_{T^k\circ\tilde f,\phi}= m_{\tilde f,\phi}+k$, for every $k\in\Z$.

Denote:
$$\lambda_{f, \mathcal{F}}(z,z'):=\lambda(z,z',\mathcal{F}, f^{-1}(\mathcal{F})),$$
\begin{equation}\label{equality}
 \Lambda_{\tilde{f},\mathcal{F},\phi}(z,z'):=\lambda_{f, \mathcal{F}}(z,z')+m_{\tilde{f},\phi}(z).
\end{equation}
\begin{defn} \enskip Fix $f\in \mathrm{Homeo}_*({D}^{*}) $ and a lift $\tilde f$ to $\tilde {{D}}$.  Say that $z\in {D}^*$ has a {\it rotation number} $\mathrm{rot}_{\tilde f}(z)\in\R$ if: 
\begin{enumerate}
\item there exists a compact set $K\subset {D}^*$ such that $\#\{n\geq 0\,\vert \, f^n(z)\in K\}=+\infty$;
\item if $\phi$ is a ray and if $K\subset {D}^*$ is a compact set containing $z$, then for every $\varepsilon>0$, there exists $n_0\geq 0$ such that for every $n\geq n_0$ it holds that
$$ f^n(z)\in K \Rightarrow \left\vert {\frac{1}{n}}\sum_{i=0}^{n-1} m_{\tilde f, \phi} (f^i(z)) -\mathrm{rot}_{\tilde f}(z)\right \vert \leq \varepsilon.$$
\end{enumerate}
\end{defn}
 \begin{rem}
The notion $\mathrm{rot}_{\tilde f}(z)$  does not depend neither on the choice of the ray, nor on the choice of the foliation. Both of these claims follow from the fact that for every compact set $K\subset D^*$ the difference in $m_{\tilde{f},\phi}(z)$ is bounded when changing either the ray of the foliation or the foliation itself (see Lemma 2.5 in \cite{Cal}).
\end{rem}
It is worth noticing the following identities, which directly follow from the definition:
\begin{align}\label{equality4}
  \sum_{i=0}^{n-1} m_{\tilde f, \phi}(f^i(z))= m_{\tilde f^n, \phi}(z),  \\
   \sum_{i=0}^{n-1} \lambda_{ f, \mathcal{F}}(f^i(z))= \lambda_{f^n, \mathcal{F}}(z),  \\
    \sum_{i=0}^{n-1} \Lambda_{\tilde f, \mathcal{F},\phi}(f^i(z))= \Lambda_{\tilde f^n, \mathcal{F},\phi}(z). 
\end{align}
In the next proposition, we consider the Euclidean radial foliation $\mathcal{F}_{*}$ on $D^{*}$, whose
leaves are the paths $$(0, 1) \ni t \to te^{2 i\pi\alpha}, \ \   \alpha \in [0, 1).$$
The second author proved the following properties in \cite{Cal}.
\begin{prop}\label{property}
Suppose that $(z,z')\in W$ and $\phi\in\mathcal{F}_{*}$. The following inequalities hold.
\begin{enumerate}
\item $    |m_{\tilde{f}^{n},\phi}(z)-W_{\tilde{f}^{n}}(0,z)|\leq 1$,
    \item $   |\Lambda_{\tilde f^{n},\mathcal{F}_{*},\phi}(z,z')-W_{\tilde f^{n}}(z,z')|\leq 2$.
\end{enumerate}
where the function $W_{\tilde{f}^{n}}$ is defined on $D\times D$ with the same formula as defined in Definition \ref{DefWindingNumber} on $\mathbb{D}\times \mathbb{D}$.
\end{prop}
\begin{defn}
    Fix $f\in \mathrm{Homeo}_{*}(D^{*})$ and a lift $\tilde{f}$ to $\tilde{D}$. Consider $\mu$ an $f$-invariant probability measure on $\D^*$,we say that $\mu$ has a rotation number $\mathrm{rot}_{\tilde{f}}(\mu) \in \mathbb{R}$ if $\mu$-almost every point $z$
has a rotation number and the function $\mathrm{rot}_{\tilde{f}}$ is $\mu$-integrable. In this case
we set
$$\mathrm{rot}_{\tilde{f}}(\mu):=\int_{{D}^{*}}\mathrm{rot}_{\tilde{f}}~d \mu.$$
\end{defn}
The following theorem is an application of property (1) of Proposition \ref{property} and equality (\ref{equality4}), which tells about the rotation number of the invariant measure.
\begin{thm}\label{integral1}
    Suppose that $f\in \Diff^{1}({D})$ fixes $0$. Let $I$ be an identity isotopy of $f$ and let $\tilde{f}$ be the lift of $f\big|_{{D}^{*}}$ to $\tilde{D}^*$ naturally defined by $\tilde I$. If $\mathcal {F} \in \mathfrak F$ is a radial foliation
and $\phi$ is a leaf of $\mathcal F$ such that $m_{\tilde f,\phi}$ is $\mu$-integrable, for some  invariant probability measure $\mu$ such that $\mu(\{0\})=0$,
then it holds that
\begin{equation}
    \mathrm{rot}_{\tilde{f}}(\mu)=\int_{{D}^{*}}W_{\tilde{f}}(0,z)~d\mu(z)=\int_{{D}^{*}}m_{\tilde{f},\phi}(z)~d\mu(z).
\end{equation}
\end{thm}
The following theorem is a generalization of Theorem \cite{Cal}. The proof uses inequality (2) of Proposition \ref{property}.
 \begin{thm}\label{integral}
    Suppose that $f\in \Diff^{1}({D})$ fixes $0$. Let $I$ be an identity isotopy of $f$ and let $\tilde{f}$ be the lift of $f\big|_{{D}^{*}}$ naturally defined by $\tilde I$. If $\mathcal{F} \in \mathfrak{F}$ is a radial foliation
and $\phi$ is a leaf of $\mathcal{F}$ such that $m_{\tilde f,\phi}$ is $\eta_{1}$ and $\eta_{2}$-integrable and $\lambda_{f,\mathcal{F}}$ is $\eta$-integrable,
then it holds that
\begin{equation}\label{meanrot}
    \int_{W}W_{\tilde{f}}(z,z')~d\eta=  \int_{W}\Lambda_{\tilde{f},\mathcal{F},\phi}(z,z')~d\eta.
\end{equation}
Here, $\eta$ is an invariant probability measure on $D\times D$, and $\eta_{1}$ and $\eta_{2}$ are the projections of $\eta$ onto the first and second factors, respectively such that $\eta(\Delta)=0$ and $\eta_{1}(\{0\})=\eta_{2}(\{0\})=0$.
\end{thm}
\begin{proof}
  By the assumptions of the theorem and Birkhoff's theorem, the sequence $\frac{m_{\tilde f^n,\phi}(z)}{n}$ converges $\eta_1$-almost everywhere and also $\eta_{2}$-almost everywhere. Moreover, the sequence
    $$\frac{\lambda_{f^n,\mathcal{F}}(z,z')}{n}$$ 
    converges $\eta$-almost everywhere on $W$. Furthermore, we deduce that the sequence $$\frac{\Lambda_{\tilde f^n,\mathcal{F},\phi}(z, z')}{n}$$
    converges $\eta$-almost everywhere. Consequently, by inequality (2) of Proposition \ref{property}, the same holds for the sequence
    $$\frac{W_{\tilde f^{n}}(z, z')}{n}$$ which converges to the same value as $$\frac{\Lambda_{\tilde f^n,\mathcal{F},\phi}(z, z')}{n}.$$ Again, by Birkhoff's theorem, we deduce the equality (\ref{meanrot}).
\end{proof}
\begin{rem}
     Recall that integrals $  \int_{W}W_{\tilde{f}}(z,z')~d\eta$ and $  \int_{W}\Lambda_{\tilde{f},\mathcal{F},\phi}(z,z')~d\eta$ depend only on the homotopy class of the isotopy with fixed ends.
\end{rem}
The following theorem is an implication of Franks result \cite{franks1988generalizations}. We could also deduce it from the more general result of the second author \cite{le2001rotation}.
\begin{thm}\label{rotation}
    If $f:\mathbb{D}\to \mathbb{D}$ is an irrational pseudo-rotation on $\mathbb{D}$, then for any invariant probability measure $\mu$ such that $\mu(\{0\})=0$  we have 
\begin{equation}\label{rotationnumber}
\int_{\mathbb{D}}W_{\tilde{f}}(0,z)~d\mu(z)=\rho(\tilde f).
\end{equation}
Where $\rho(\tilde f)$ is the rotation number of the irrational pseudo-rotation $f$.
\end{thm}
\section{Radial foliations for irrational pseudo-rotations and their properties} Let $f$ be an irrational pseudo rotation and $\tilde f$ an identity
isotopy of $f$. If $f$ is extended to the whole plane, the isotopy
$\tilde f$ defines a natural lift of $f_{{\mathbb R}^2\setminus\{0\}}$
to the universal covering space of ${\mathbb {R}^2\setminus\{0\}}$, also
denoted $\tilde f$.  If $T$ is the natural generator of the group of
covering transformations, every map $\tilde f^b\circ T^{-a}$, $b\in
\mathbb N\setminus\{0\}$, $a\in \mathbb Z$ is associated to a certain
identity isotopy of $f^b$ denoted $\tilde f^b\circ T^{-a}$.

In this section, we will explain how after suitably extending $f$ to the
whole plane, we can construct for every couple $(a,b)$ an invariant disk
$\mathbb{D}_{a/b}$ containing $\mathbb D$ and a radial foliation $\mathcal F$
on $\D_{a/b}$ such that:
\begin{itemize}[label=-]
    \item the winding distance between $\mathcal F$ and $f^{-b}(\mathcal F)$ is
bounded by $Mb$, where $M$ depends only on $f$,
\item every leaf of the lifted foliation $\mathcal F$ to the universal
covering space of $\D_{a/b}\setminus\{0\}$ is sent on its right by
$\tilde f^b\circ T^{-a}$.
\end{itemize}
We will use the fact that the rotation number of a probability measure
$\mu$ invariant by the original map $f$ is equal to the measure of the
strip between a leaf of $\tilde {\mathcal {F}}$ and its image by $\tilde f^b\circ
T^{-a}$
and the bound above concerning foliations to give an explicit upper
bound of the integral of the winding function of $\tilde f^b\circ
T^{-a}$ with respect to the product of any two invariant  probability
measures. Using this upper bound and using a sequence of convergents of
$\tilde \alpha$, we will deduce that the measure of the winding function
of $\tilde f$ with respect to the product of these two invariant 
probability measures is equal to $\rho(\tilde f)$.

Assume $f$ is a $C^{1}$ is an irrational pseudo-rotation of $\mathbb{D}$ that fixes the center $0=(0,0)$ of $\mathbb{D}$
and coincides with a rotation $R_{\alpha}$ on $\partial \mathbb{D}$. Fix $\beta> \alpha$. We can extend our map to a
homeomorphism of the whole plane (also denoted $f$) such that:
\[ f(x) = \begin{cases} 
    R_{\alpha+|z|-1}(z) & \text{if }  1\leq |z|\leq 1+\beta-\alpha, \\
     R_{\beta} & \text{if } |z|\geq 1+\beta-\alpha. 
   \end{cases}\]
We suppose moreover than $\beta\notin \mathbb{Q}$ and that $(\alpha, \beta)\cap \mathbb{Z} =\emptyset$. The extension $f$ is a piecewise $C^{1}$ area preserving transformation that satisfies the following properties:
\begin{itemize}
    \item $0$ is the unique fixed point of $f$;
    \item there is no periodic point of period $b$ if $(b\alpha, b\beta) \cap \mathbb{Z} = \emptyset$;
    \item if  $(b\alpha, b\beta) \cap \mathbb{Z} \neq \emptyset$, the set of periodic points of period $b$ can be written $\cup_{\alpha<a/b<\beta} S_{a/b}$, where $S_{a/b}$ is the circle of center $0$ and radius $1 + a/b-\alpha$. 
\end{itemize}
We denote the disk centered at $0$ and with radius disk $1 + a/b-\alpha$ by $\mathbb{D}_{a/b}$. 

Using generating functions and the decomposition of $f$ in $m$ maps ``close to the identity'', $m\geq 1$, the second author constructed a radial foliation $\mathcal{F}_{1}$ on $\mathbb D_{a/b}^* =\mathbb D_{a/b}\setminus\{0\}$ and an isotopy ${I}=({f}_{s})_{s\in [0,mb]}$ from $\mathrm{id}$ to $f^{b}$ for every $b$ such that $f^{b}|_{S_{a/b}}=\mathrm{id}$ and $f^{b}(z)\neq z$ if $1\leq |z|< 1+a/b-\alpha$.  The isotopy ${I}$ fixes each point of $S_{a/b}$. We do not provide the construction in detail, but instead utilize the properties of the foliation and the isotopy. For more details, please refer to \cite{Cal}. 
\begin{prop}\cite[Corollary ~3.13]{Cal}\label{cor} For every $z,z'\in  \mathbb{D}^{*}_{a/b}$ the following holds:
\begin{equation*}
   {\tau}(\tilde{z},\tilde{z}',\mathcal{F}_{1},f^{-b}(\mathcal{F}_{1}))\leq 4mb.
\end{equation*}
\end{prop}
\begin{prop}\cite[Proposition ~3.14]{Cal}\label{brouwer}
If $\tilde{{I}} =(\tilde{{f}}_{s})_{s\in[0,2mb]}$ is the identity isotopy on the universal covering space $\tilde {\mathbb D}_{a/b}$ of $\mathbb D_{a/b}^*$ that lifts $({f}_{s}\big|_{\mathbb{D}^{*}_{a/b}})_{s\in [0,2mb]}$, then every leaf $\tilde{\phi}$ of $\tilde{\mathcal{F}_{1}}$ is a Brouwer line of $\tilde{{f}}_{2mb}$. That is $\overline{R(\tilde{{f}}_{2mb}(\tilde{\phi}_{\tilde{z}}))}\subset R(\tilde{\phi}_{\tilde{z}})$ for every $\tilde{z}\in \mathbb{D}^{*}_{a/b}$.
\end{prop}
We will prove now the following result. Such kind of result was established in \cite{Cal}, but with the Lebesgue measure on the whole disk $D_{a/b}$.
\begin{lem}\label{mainclaim2}
   We fix $z\in \mathbb{D}^{*}_{a/b}$  and choose a lift $\tilde{z}\in \widetilde{\mathbb{D}}_{a/b}$. Now we define 
   \[O(\tilde{z}):=\overline{R(\tilde{{f}}_{2mb}^{-1}(\tilde{\phi}_{\tilde{z}}))}\cap L(\tilde{\phi}_{\tilde{z}}).\]
   For every invariant probability measure $\mu$ such that $\mu(\{0\})=0$ and $\mu(\mathbb{D})=1$,  we have
 \begin{equation}\label{measurebound}
        \tilde{\mu}(O(\tilde{z}))=a-b\alpha.
 \end{equation}
 Here, the measure $\tilde{\mu}$ is a lift of $\mu$.
\end{lem}
\begin{rem}
The number $\mu(O(\tilde{z}))$ is independent of $z$ and its lift; it is an invariant of the map, meaning it is independent of the isotopy with fixed ends, and of the invariant measure. In the proof, we show that it is also independent of the invariant probability measure.  
\end{rem}
\begin{proof} 
Without loss of generality, we assume $\mu$ is an invariant probability measure of $f:\mathbb{R}^{2}\to\mathbb{R}^{2}$ that satisfies the condition $\mu(\{0\})=0$ and $\mu(\mathbb{D})=1$, 
The set $\tilde{\mathbb{D}}_{a/b}$ has the following form:
\[
\tilde{\mathbb{D}}_{a/b}=\mathbb{R}\times(0,1-a/b+\alpha].
\]
In this case the fundamental covering automorphism $T$ can be explicitely written as
$$T(x,y):=(x+2\pi,y).$$
Finally define the projection map 
\[
p_{1}:\tilde{\mathbb{D}}_{a/b}\to \mathbb{R} \qquad (x,y)\mapsto p_1(x,y)=x.
\]
By Theorem \ref{rotation}, $\mu$ has a rotation number $\alpha$ for the lift $\tilde{f}$ to $\tilde{\mathbb{D}}_{a/b}$ given by the decomposition of $f$. Therefore, it has a rotation number 
$$b\alpha-a$$ 
for the lift $\tilde{{f}}_{2mb}= \tilde{f}^b\circ {T}^{-a}$ of $f^b$ that fixes the points of the preimage of $S_{a/b}$ by the covering projection. Let us explain why $a-b\alpha$ is the measure of the domain between a leaf $\tilde \phi$ (included) and $\tilde{{f}}_{2mb}(\tilde\phi)$ (excluded).

 Consider the first coordinate projection $p_1$ and the map $\pi_1$ that takes the values $k$ in the domain contained between $T^{k}(\tilde\phi)$ (included) and $T^{k+1}(\tilde{\phi})$ (excluded). 

The real functions $\pi_1\circ \tilde{f}_{2mb} -\pi_1$, $p_1\circ \tilde{f}_{2mb} -p_1$, $p_1-\pi_1$, being invariant by $T$, can be seen as defined on $\mathbb{D}^*_{a/b}$ and that we will use later the equality
$$\pi_1\circ \tilde{ f}^{i+1}_{2mb} -\pi_1\circ\tilde{ f}^i_{2mb}= (\pi_1\circ \tilde{ f}_{2mb} -\pi_1)\circ f^i_{2mb}$$ and similar ones for the other maps.
 
We have:
\[\int_{\mathbb{D}^{*}_{a/b}} (p_1\circ \tilde{{f}}_{2mb}-p_1 )~d{\mu} =-(a-b\alpha)\]
by definition of the rotation vector. We also know that
\begin{equation}\label{equalityarea}
    \int_{\mathbb{D}^{*}_{a/b}} (\pi_{1}\circ \tilde{{f}}_{2mb}-\pi_{1}) ~d{\mu} = -\tilde \mu(O(\tilde{z})).
\end{equation}
To see why this holds, note that $\pi_{1}\circ
\tilde{{f}}_{2mb}-\pi_{1}$ takes non positive values. Moreover, for
every point $x\in \mathbb{D}^*_{a/b}$, the fact that $(\pi_{1}\circ
\tilde{{f}}_{2mb}-\pi_{1})(x) =-k<0$ means that there exists a lift
$\tilde x$ of $x$ such that
$$T(\tilde x)\in L(\tilde f^{-1}_{2mb}(\tilde\phi_{\tilde z})),\,\tilde
x\in \overline{R(\tilde f^{-1}_{2mb}(\tilde\phi_{\tilde z}))},\dots
,\,T^{-k+1}(\tilde x)\in \overline{R(\tilde
f^{-1}_{2mb}(\tilde\phi_{\tilde z}))},\, T^{-k}(\tilde x)\in \overline
{R(\tilde\phi_{\tilde z})}.$$ So we have
$$\tilde \mu(O(\tilde{z}))=\sum_{k>0} k\mu(\{x\in
\mathbb{D}^*_{a/b}\,\vert \enskip(\pi_{1}\circ
\tilde{{f}}_{2mb}-\pi_{1})(x) =-k\})= -\int_{\mathbb{D}^{*}_{a/b}}
(\pi_{1}\circ \tilde{{f}}_{2mb}-\pi_{1}) ~d{\mu}.$$
By Birkhoff's ergodic theorem for $\mu$-almost every $z$, the sequences 
\[\frac{\sum_{i=0}^{n-1}(p_1\circ \tilde{{f}}_{2mb}-p_1)\circ f^{i}_{2mb}(z)}{n} \ \ \mbox{and} \ \ \frac{\sum_{i=0}^{n-1}(\pi_{1}\circ \tilde{{f}}_{2mb}-\pi_{1})\circ f^{i}_{2mb}(z)}{n}\]
converge, the first one to $p_{*}(z)\in \mathbb{R}$ the second one to $\pi_{*}(z)\in[-\infty, 0].$ Moreover, 
\begin{itemize}
    \item $p_{*}$ is integrable and $\int_{\mathbb{D}^{*}_{a/b}} p_{*}~d\mu=\int_{\mathbb{D}^{*}_{a/b}} p_{1}\circ \tilde{f}_{2m}-p_{1}~d\mu,$
    \item $\int_{\mathbb{D}^{*}_{a/b}} \pi_{*}~d\mu=\int_{\mathbb{D}^{*}_{a/b}} \pi_{1}\circ \tilde{f}_{2m}-\pi_{1}~d\mu.$
\end{itemize}
We know that $\mu$-almost every point is recurrent and that the function $(p_1-\pi_1)$ is locally bounded.  This implies $\pi_{*}(z)=p_{*}(z)$ for $\mu$-almost every $z$ because when we consider the Birkhoff means of the two functions applied between $0$ and $n_k$, where ${f}_{2mb}^{n_k}(z)$ is close to $z$, we find that
\begin{align*}
  p_{*}(z)-\pi_{*}(z) &=\lim_{k\to \infty}\frac{\sum_{i=0}^{n_{k}-1}(p_1\circ \tilde{{f}}_{2mb}-p_1)\circ f^{i}_{2mb}(z)-(\pi_{1}\circ \tilde{{f}}_{2mb}-\pi_{1})\circ f^{i}_{2mb}(z)}{n_{k}}\\
   &=\lim_{k\to \infty}\frac{\sum_{i=0}^{n_{k}-1}(p_1\circ \tilde{{f}}_{2mb}^{i+1}-p_1\circ \tilde{{f}}_{2mb}^{i})(z)-(\pi_{1}\circ \tilde{{f}}_{2mb}^{i+1}-\pi_{1}\circ \tilde{{f}}_{2mb}^{i})(z)}{n_{k}}\\
   &=\lim_{k\to \infty}\frac{\sum_{i=0}^{n_{k}-1}(p_1\circ \tilde{{f}}_{2mb}^{i+1}-\pi_{1}\circ \tilde{{f}}_{2mb}^{i+1})(z)-(p_1\circ \tilde{{f}}_{2mb}^{i}-\pi_{1}\circ \tilde{{f}}_{2mb}^{i})(z)}{n_{k}}\\
   &=\lim_{k\to \infty}\frac{(p_1\circ \tilde{{f}}_{2mb}^{n_{k}}-\pi_{1}\circ \tilde{{f}}_{2mb}^{n_{k}})(z)-(p_1-\pi_{1})(z)}{n_{k}}\\
   &=\lim_{k\to \infty}\frac{(p_1-\pi_{1})\circ f^{n_{k}}_{2mb}(z)-(p_1-\pi_{1})(z)}{n_{k}}=0
\end{align*}
for $\mu$-almost every $z$. This completes the proof of Lemma \ref{mainclaim2}.
\end{proof}
\begin{rem}
We could attempt to prove Lemma \ref{mainclaim2} using the following approach:
\begin{align*}
\int_{\mathbb{D}^{*}_{a/b}}  (p_1\circ \tilde{{f}}_{2mb} - p_1) ~d\mu - \int_{\mathbb{D}^{*}_{a/b}}  (\pi_{1}\circ \tilde{{f}}_{2mb} - \pi_{1})~ d\mu \\
= \int_{\mathbb{D}^{*}_{a/b}}  (p_1\circ \tilde{{f}}_{2mb} - \pi_1\circ \tilde{{f}}_{2mb}) - (p_1 - \pi_1))~ d\mu\ \\= \int_{\mathbb{D}^{*}_{a/b}}  ((p_1 - \pi_{1})\circ \tilde{{f}}_{2mb} - (p_1 - \pi_1))~ d\mu\
&= 0.
\end{align*}
This might seem correct at first glance.  However, the last equality is not correct because one does not know if $p_1-\pi_1$
 is integrable, which is necessary to use the fact that $\mu$ is invariant.
\end{rem}
\begin{rem}
    Lemma \ref{mainclaim2} holds for any Brouwer radial foliation, not just for the foliation constructed by the second author in \cite{Cal}.
\end{rem}
We now introduce the following notation 
 $$W_{a/b}:= \{(z, z') \in \mathbb{D}^{*}_{a/b} | \  z\neq z^{'} \}.$$
 The following result is a generalization of Lemma 4.2 in \cite{Cal} for every couple of measures instead of the product of the Lebesgue measure by itself. 
 \begin{lem}\label{bound2}
For any two invariant probability measures $\nu$ and $\mu$ of  $f$ such that $\mu(\{0\})=\nu(\{0\})=0$ and $\mu(\mathbb{D})=\nu(\mathbb{D})=1$, the following holds.
 \begin{equation*}
\int_{W_{a/b}}\overline{{\tau}}(z,z',\mathcal{F}_{1},f^{-b},(\mathcal{F}_{1}))~d\mu(z)\times d \nu(z')\leq 8 mb(a-b\alpha).
\end{equation*}
\end{lem}
\begin{rem}
 The proof of Lemma \ref{bound2} is identical to the proof of Lemma 4.2 in \cite{Cal}. It relies on Proposition \ref{cor} and the identity \ref{measurebound} for the Lebesgue measure. However, based on Lemma \ref{mainclaim2}, we know that this identity \ref{measurebound} holds for any two invariant probability measures. Therefore, we omit it
\end{rem}
The following theorem is a generalization of Theorem 1.1 \cite{Cal} for any pair of invariant measures.
\begin{thm}\label{intmeas}
   For any two invariant probability measures $\nu$ and $\mu$ of $f\big|_{\mathbb{D}}$ such that $\mu(\{0\})=\nu(\{0\})=0$. We have the following
 \begin{equation}\label{identity1}
     \int_{\mathbb{D}\times \mathbb{D}}W_{\tilde{f}}(x,y)~d\mu\times d\nu=\rho(\tilde{f}).
 \end{equation}
\end{thm}
\begin{rem}
 Recall that the right-hand side of the identity \ref{identity1} depends only on the homotopy class of the isotopy with fixed ends that we choose, similarly the left-hand side changes accordingly.
\end{rem}
\begin{proof}[Proof of Theorem \ref{intmeas}]
First of all we extend the invariant probability measures $\mu$ and $\nu$ to $\mathbb{D}_{a/b}$ in the following way:
\[
\mu(A)=\mu(\mathbb{D}\cap A) \quad \text{and} \quad \nu(A)=\nu(\mathbb{D}\cap A)
\]
for every Borel set $A$ of $\mathbb{D}_{a/b}$.

Using Lemma \ref{bound2} and inequality (\ref{inequality}), we establish the integrability of $\lambda(z,z',\mathcal{F}_{1},f^{-b}(\mathcal{F}_{1}))$. In particular we have
\[
\left|\int_{W_{a/b}}\lambda(z,z',\mathcal{F}_{1},f^{-b}(\mathcal{F}_{1}))~d\mu(z)\times d \nu(z')\right|\leq 8 mb (a-b\alpha).
\]
By the last inequality together with equality (\ref{equality}), allows us to deduce that 
\[
\left|\int_{W_{a/b}}\Lambda_{\tilde{f}_{2mb},\mathcal{F}_{1},\phi}(z,z') -m_{\tilde{f}_{2mb},\phi}(z')~d \mu(z)\times d \nu(z')\right|\leq 8 mb (a-b\alpha)
\]
if $\phi$ is a leaf of $\mathcal {F}_1$.

By Theorem \ref{rotation}, we know that $m_{\tilde{f}_{2mb},\phi}(z)$ is integrable, and this, together with the last inequality, implies the following:
\[
\left|\int_{W_{a/b}}\Lambda_{\tilde{f}_{2mb},\mathcal{F}_{1},\phi}(z,z') ~d \mu(z)\times d \nu(z')-\int_{W_{a/b}}m_{\tilde{f}_{2mb},\phi}(z')~d \mu(z)\times d \nu(z')\right|\leq
\]
$$\leq 8 mb (a-b\alpha).$$
The fact that  $\mu(\{0\})=\nu(\{0\})=0$ implies that  $$\mu\times \nu(\{(z,z')\in \mathbb{D}^{*}_{a/b}\times\mathbb{D}^{*}_{a/b}| \ z=z'\})=0,$$ 
using this together with previous inequality  yields
\[
\left|\int_{\mathbb{D}^{*}_{a/b}\times \mathbb{D}^{*}_{a/b}}\Lambda_{\tilde{{f}}_{2mb},\mathcal{F}_{1},\phi}(z,z')~ d \mu(z)\times d \nu(z')-\int_{\mathbb{D}^{*}_{a/b}\times \mathbb{D}^{*}_{a/b}}m_{\tilde{f}_{2mb},\phi}(z')~d \mu(z)\times d \nu(z')\right|\leq
\]
$$\leq 8 mb (a-b\alpha).$$
Combining this with Theorem  \ref{integral1} and \ref{integral}, we can now state the following result:
\[\left|\int_{\mathbb{D}^{*}_{a/b}\times \mathbb{D}^{*}_{a/b}}W_{\tilde{f}_{2mb}}(x,y) ~ d \mu(z)\times d \nu(z')-\int_{\mathbb{D}^{*}_{a/b}}\text{rot}_{\tilde{f}_{2mb}}(\nu)~d \mu(z)\right|\leq
\]
$$\leq 8 mb (a-b\alpha).$$
We also know that ${f}_{2mb}=\tilde{f}^{b}\circ T^{-a}$ and so we obtain 
\begin{align*}
   \left| b\int_{\mathbb{D}^{*}_{a/b}\times \mathbb{D}^{*}_{a/b}}W_{\tilde{f}}(x,y)~d \mu(z)\times d \nu(z') -a-b\int_{\mathbb{D}^{*}_{a/b}}\text{rot}_{\tilde{f}}(\nu)~d \mu(z)+a\right|\leq 8b m (a-b\alpha)
   \end{align*}
which can be written as
\[\left|\int_{\mathbb{D}^{*}_{a/b}\times \mathbb{D}^{*}_{a/b}}W_{\tilde{f}}(x,y)~ d \mu(z)\times d \nu(z')-\alpha\right|\leq 8 m (a-b\alpha).\]
One can find a sequence $(a_{n}, b_{n})_{n\geq 0}$ such that $\lim_{n\to + \infty} a_{n}-b_{n}\alpha = 0.$ Writing
\[\left|\int_{\mathbb{D}^{*}_{a_{n}/b_{n}}\times \mathbb{D}^{*}_{a_{n}/b_{n}}}W_{\tilde{f}}(x,y) ~d \mu(z)\times d \nu(z')-\alpha\right|\leq 8 m  (a_{n}-b_{n}\alpha). \]
By construction of $\mu$ and $\nu$, we have
\[
\left|\int_{\mathbb{D}\times \mathbb{D}}W_{\tilde{f}}(x,y)~ d \mu(z)\times d \nu(z')-\alpha\right|\leq 8 m  (a_{n}-b_{n}\alpha). 
\]
and letting $n$ tend to $+\infty$, one obtains 
\[
\int_{\mathbb{D}\times \mathbb{D}}W_{\tilde{f}}(x,y)~ d \mu(z)\times d \nu(z')-\alpha=0.
\]
This completes the proof of Theorem \ref{intmeas}.
\end{proof}
\section{Proof of Main Results}
In this section, we use the last result from the previous section together with a Krylov-Bogolyubov type argument to prove Theorems \ref{asymp-mean} and \ref{asym-link}.

In his thesis, the first author established the inequality (\ref{actionbound}) for $C^{\infty}$ smooth area-preserving diffeomorphisms, but it was not stated as a theorem. However, one could find it in the proof of Theorem 4.3.2 in \cite{BecharaSenior2023}, which is a generalization of the main argument in \cite{bechara2021asymptotic}. By repeating the same argument, we could also show that it works for $C^{1}$ area-preserving diffeomorphisms. As the proof remains the same, we omit its details. 
\begin{thm}{\cite{BecharaSenior2023}}\label{bech}
If $f:\mathbb{D}\to \mathbb{D}$ is a $C^{1}$ area-preserving diffeomorphism, then 
\begin{equation}\label{actionbound}
    \Big|\frac{a_{\tilde{f}^{n}}(x)-\int_{\mathbb{D}}W_{\tilde{f}^{n}}(x,y)~d\omega(y)}{n}\Big|\leq \frac{8}{n}.
\end{equation}
Here  $\omega=\frac{r}{\pi}dr\wedge d\theta$.
\end{thm}
We will now prove Theorem \ref{asymp-mean} by invoking Theorem \ref{bech} and Theorem \ref{intmeas}.
\begin{proof}[Proof of Theorem \ref{asymp-mean}]
Theorem \ref{bech} implies the following claim:

\underline{Claim}: For any invariant probability measure $\mu$, we have 
\[
\int_{D}a_{\tilde{f}}(x)~d\mu(x)=\rho(\tilde{f}).
\]
\begin{claimproof}
  Using inequality (\ref{actionbound}), we obtain the following inequality:
  \[
  -\frac{8}{n}\leq\frac{a_{\tilde{f}^{n}}(x)-\int_{\mathbb{D}}W_{\tilde{f}^{n}}(x,y)~d\omega(y)}{n}\leq \frac{8}{n}.
  \] 
  First, we integrate both sides of the inequality with respect to the measure $\mu$, and then apply the following two identities: 
  \begin{equation}\label{formula3}
      \int_{\mathbb{D}}a_{\tilde{f}} (x)~d\mu(x)=\frac{\int_{\mathbb{D}}a_{\tilde{f}^{n}}(x)~d\mu(x)}{n},
  \end{equation}
  \begin{equation}\label{formula5}
\int_{\mathbb{D}}\int_{\mathbb{D}}W_{\tilde{f}}(x,y)~d\omega(y)\times d\mu(x)=\frac{\int_{\mathbb{D}}\int_{\mathbb{D}}W_{\tilde{f}^{n}}(x,y)~d\omega(y)\times d\mu(x)}{n}. 
  \end{equation}
Therefore, we have:
  \[
  -\frac{8}{n}\leq\frac{\int_{\mathbb{D}}a_{\tilde{f}^{n}}~d\mu(x)-\int_{\mathbb{D}}\int_{\mathbb{D}}W_{\tilde{f}^{n}}(x,y)~d\omega(y)\times d\mu(x)}{n}\leq \frac{8}{n}.
  \]
  The last inequality together with the equalities (\ref{formula3}) and (\ref{formula5}) implies that
  \[
  -\frac{8}{n}\leq\int_{\mathbb{D}}a_{\tilde{f}}~d\mu(x)-\int_{\mathbb{D}}\int_{\mathbb{D}}W_{\tilde{f}}(x,y)~d\omega(y)\times d\mu(x)\leq \frac{8}{n}.
  \]
By taking the limit from both sides of the inequality, we obtain:
\begin{equation}\label{formula6}
\int_{\mathbb{D}}a_{\tilde{f}}~d\mu(x)=\int_{\mathbb{D}}\int_{\mathbb{D}}W_{\tilde{f}}(x,y)~d\omega(y)\times d\mu(x).
\end{equation}
We can write the measure $\mu$ in the following form $b\delta_{0}+(1-b)\mu_{1}$ where $\delta_{0}$ is the Dirac measure at $0$, $b\in [0,1]$, and the measure $\mu_{1}$ satisfies $\mu_{1}(\{0\})=0$. Now we rewrite the equality (\ref{formula6}) using it.
$$\int_{\mathbb{D}}a_{\tilde{f}}~d\mu(x)=b\int_{\mathbb{D}}\int_{\mathbb{D}}W_{\tilde{f}}(x,y)~d\omega(y)\times d\delta_{0}(x)+(1-b)\int_{\mathbb{D}}\int_{\mathbb{D}}W_{\tilde{f}}(x,y)~d\omega(y)\times\mu_{1}(x).$$
By Theorem \ref{intmeas} and the equality (\ref{formula6}), we have the following.
\begin{equation}\label{formula}
\int_{\mathbb{D}}a_{\tilde{f}}~d\mu(x)=b\int_{\mathbb{D}}a_{\tilde{f}}~d\delta_{0}(x)+(1-b)\rho(\tilde{f}).
\end{equation}
By combining of Theorem \ref{bech} and Theorem \ref{rotation}, we get $\int_{\mathbb{D}}a_{\tilde{f}}d\delta_{0}(x)=\rho(\tilde{f})$, together with the last equality this gives the desired equality $\int_{\mathbb{D}}a_{\tilde{f}}d\mu(x)=\rho(\tilde{f})$.
\end{claimproof}

Assume the sequence \[
\left(\frac{\sum^{n-1}_{i=0}a_{\tilde{f}}(f^{i}(x))}{n}\right)_{n\geq 0}
\]
does not converge uniformly, so there exist
$\epsilon>0$ and a sequence $x_{{k}}$ such that 
\begin{equation}\label{ineq1}
   \big|\frac{\sum^{n_{k}-1}_{i=0}a_{\tilde{f}}(f^{i}(x_{{k}}))}{n_{k}}-\rho(\tilde{f})\big|\geq \epsilon.
\end{equation}
We now choose a subsequence $n_{k_{j}}$ of $n_{k}$ so that 
$$\mu_{n_{k_{j}}}(x_{k_{j}})\rightharpoonup \mu,$$ 
here,  $\mu$ is an invariant probability measure of $f:\mathbb{D}\to \mathbb{D}$. The definition of weak$^{*}$ convergence implies the following identity:
\[
\lim_{k_{j}\to \infty}\frac{\sum^{n_{k_{j}}-1}_{i=0}a_{\tilde{f}}(f^{i}(x_{{k_{j}}}))}{n_{k_{j}}}=\lim_{{k_{j}}\to \infty}\int_{\mathbb{D}}a_{\tilde{f}} \ d\mu_{n_{k_{j}}}(x_{k_{j}})=\int_{\mathbb{D}}a_{\tilde{f}} \ d\mu=\rho(\tilde{f}).
\]
 The last equality contradicts the inequality (\ref{ineq1}). This completes a proof of Theorem \ref{asymp-mean}.
\end{proof}
We can now prove Theorem \ref{asym-link} as an application of Theorem \ref{intmeas}.
\begin{proof}[Proof of Theorem \ref{asym-link}]
Let us begin by assuming that the sequence
\[
\int_{\mathbb{D}\times\mathbb{D}}W_{\tilde{f}}(x,y)~d\mu_{n}(x)\times d\mu_{n}(y)=\frac{\sum_{i=0}^{n-1}\sum_{j=0}^{n-1}W_{\tilde{f}}(f^{i}(x),f^{j}(y))}{n^{2}}
\]
does not converge uniformly on $ \mathbb{O}(\mathbb{D})$ to $\rho(\tilde f)$. This implies the existence of $\epsilon>0$, of an increasing sequence $(n_k)_{k\geq 0}$ of positive integers and of a sequence $(x_k, y_k)_{k\geq 0}$ in $\mathbb{D}\times \mathbb{D}$ such that
\begin{equation}\label{ineq2}
\left|\frac{\sum_{i=0}^{n_{k}-1}\sum_{j=0}^{n_{k}-1}W_{\tilde{f}}(f^{i}(x_{{k}}),f^{j}(y_{{k}}))}{n_{k}^{2}}-\rho(\tilde{f})\right|\geq \epsilon.
\end{equation}
By Proposition \ref{invmeas1}, we can find a subsequence $({k_{j}})_{j\geq 0}$ such that
\[
\mu_{n_{k_{j}}}(x_{k_j}) \rightharpoonup \mu 
\]
and 
\[
\mu_{n_{k_{j}}}(y_{k_j})\rightharpoonup \nu
\]
where $\mu$ and $\nu$ are invariant probability measures.

We may assume, up to passing to a further sub-sequence that the sequence $(n_{k_{j}})_{j\in \N}$ is such that $\mu_{n_{k_{j}}}(x_{k_{j}})\times \mu_{n_{k_{j}}}(y_{k_{j}})$ weakly  converges in $K$,  we denote by $\lambda$ its limit. Here $K$ is the compactification of $\mathbb{D}\times \mathbb{D}\setminus\Delta$ such that $W_{\tilde f}(z,z')$ is well-defined and continuous 
 (see Defition \ref{DefWindingNumber}), where $\Delta$ is the diagonal.

Define $\pi:K\to \mathbb{D}\times \mathbb{D}$ as the (continuous) projection, which coincides with the inclusion on $(\mathbb{D}\times \mathbb{D})\setminus\Delta$ and maps $(x, \xi)\mapsto (x,x)$ on $\partial K=T^{1}\mathbb{D}$. We have $\pi_*(\lambda)=\mu\times\nu$, and therefore $\lambda(\partial K)=0 $ if and only if $\mu\times\nu(\Delta)=0$.

Define $\rho'=\int_K W_{\tilde f} \, d\lambda$ and note that $\vert\rho'-\rho(\tilde f)\vert \geq \varepsilon$ because 
$$\int_K W_{\tilde f} \,d\lambda =\lim_{j\to +\infty} \int_KW_{\tilde f} \, d(\mu_{n_{k_j}}(x_{k_j})\times \mu _{n_{k_j}}(y_{k_j}))=\lim_{j\to +\infty} \int_{\mathbb{D}\times \mathbb{D}}W_{\tilde f} \, d(\mu_{n_{k_j}}(x_{k_j})\times \mu _{n_{k_j}}(y_{k_j})) $$
$$
=\lim_{j\to+\infty} \frac{\sum_{i=0}^{n_{k_j}-1}\sum_{j=0}^{n_{k_j}-1}W_{\tilde{f}}(f^{i}(x_{{k_j}}),f^{j}(y_{{k_j}}))}{n_{k_j}^{2}}$$ 
where each element of the sequence in the last limit satisfies (\ref{ineq2}).

In the case where $\lambda(\partial K)=0$, we can write
\[
\rho'=\int_{K} W_{\tilde{f}}~ d\lambda =\int_{K\setminus \partial K} W_{\tilde{f}} ~d\lambda= \int_{(\mathbb{D}\times \mathbb{D})\setminus\Delta} W_{\tilde{f}}~ d(\mu\times\nu)=\int_{\mathbb{D}\times \mathbb{D}} W_{\tilde{f}}~ d(\mu\times\nu).
\]
By Theorem \ref{intmeas}, we have the following: 
$$\int_{\mathbb{D}\times \mathbb{D}} W_{\tilde{f}}~ d(\mu\times\nu)= \rho(\tilde{f}).$$
This leads to a contradiction.

If $\lambda(\partial K)\neq 0 $, then $\mu\times\nu(\Delta)\not=0$, indicating that $\mu$ and $\nu$ share a common atom. However, the only invariant atomic measure is the Dirac measure $\delta_0$ because $f$ is an irrational pseudo-rotation. So, we can express $\mu=t\mu'+(1-t)\delta_0$ and $\nu=s\nu'+(1-s)\delta_0$, where $\nu'$ and $\mu'$ do not have atoms and $$\lambda= st\,\mu'\times\nu'+t(1-s) \,\mu'\times\delta_0+(1-t) \,s\delta_0\times \nu'+(1-t)(1-s)\lambda',$$where $\mu'(\{0\})= \nu'(\{0\})=0$ and $\lambda'(T^1_0 (\mathbb{D}))=1$ (here $T^1_0 (\mathbb{D})$ denotes the unit tangent plane of $\mathbb{D}$ at $0$).

For the same reasons as before, it follows that
\begin{align*}
    \int_{K}W_{\tilde{f}}~ d\lambda &= st \int_{\mathbb{D}\times \mathbb{D}} W_{\tilde{f}}~ d(\mu'\times\nu')+ t(1-s) \int_{\mathbb{D}\times \mathbb{D}} W_{\tilde{f}}~d(\mu'\times\delta_0) \\
                &\quad+   (1-t) s \int_{\mathbb{D}\times \mathbb{D}} W_{\tilde{f}}~d(\delta_0\times\nu')+(1-t)(1-s)\int_{\partial K} W_{\tilde{f}}~d\lambda'. 
\end{align*}
Here, the first three integrals are equal to $\rho(\tilde{f})$. To complete the contradiction, we need to show that the last integral is also equal to $\rho(\tilde{f})$. This is evident because the rotation number induced on the unit tangent vector by the differential is $\rho(\tilde{f})$. This concludes the proof of Theorem \ref{asym-link}.
\end{proof}
\bibliographystyle{amsalpha}
\bibliography{main}
\end{document}